\documentclass[11pt]{article} 

\usepackage{amsmath}
\usepackage{amsthm}
\usepackage{amsfonts}
\usepackage{mathrsfs}
\usepackage{stmaryrd}
\usepackage{setspace}
\usepackage{fullpage}
\usepackage{amssymb}
\usepackage{breqn}
\usepackage{enumitem}
\usepackage{bbold} 
\usepackage{authblk}
\usepackage{comment}
\usepackage{hyperref}
\usepackage{pgf,tikz}
\usepackage{graphicx}
\usepackage{colonequals}
\usepackage{microtype}
\usepackage{cite}
\usetikzlibrary{decorations.pathreplacing,arrows}
\tikzstyle{vertex}=[circle,draw=black,fill=black,inner sep=0,minimum size=3pt,text=white,font=\footnotesize]

\newtheorem{definition}{Definition}[section]
\newtheorem{claim}{Claim}

\newtheorem{theorem}[definition]{Theorem}

\newtheorem{lemma}[definition]{Lemma}

\newtheorem{question}[definition]{Question}

\numberwithin{equation}{section}

\newcommand{\ignore}[1]{}

\if11     
\usepackage[mathlines]{lineno}
\newcommand*\patchAmsMathEnvironmentForLineno[1]{%
  \expandafter\let\csname old#1\expandafter\endcsname\csname #1\endcsname
  \expandafter\let\csname oldend#1\expandafter\endcsname\csname end#1\endcsname
  \renewenvironment{#1}%
     {\linenomath\csname old#1\endcsname}%
     {\csname oldend#1\endcsname\endlinenomath}}%
\newcommand*\patchBothAmsMathEnvironmentsForLineno[1]{%
  \patchAmsMathEnvironmentForLineno{#1}%
  \patchAmsMathEnvironmentForLineno{#1*}}%
\AtBeginDocument{%
\patchBothAmsMathEnvironmentsForLineno{equation}%
\patchBothAmsMathEnvironmentsForLineno{align}%
\patchBothAmsMathEnvironmentsForLineno{flalign}%
\patchBothAmsMathEnvironmentsForLineno{alignat}%
\patchBothAmsMathEnvironmentsForLineno{gather}%
\patchBothAmsMathEnvironmentsForLineno{multline}%
}

\fi

\def\inst#1{$^{#1}$}

\begin{document}

\title{The number of $k$-dimensional corner-free subsets of grids}
\author{Younjin Kim\inst{1}
\thanks{The author was supported by Basic Science Research Program through the National Research Foundation of Korea(NRF) funded by the Ministry of Education (2017R1A6A3A04005963)}}

\maketitle

\begin{center}
{\footnotesize
\inst{1} 
Department of Mathematics, POSTECH, Pohang, South Korea \\
\texttt{mathyounjinkim@gmail.com}
\\\ \\
}
\end{center}

\begin{abstract}
A subset $A$ of the $k$-dimensional grid $\{1,2, \dots, N\}^k$  is said to be $k$-{\it dimensional corner-free} if it does not contain a set of points of the  
form $\{ \bold{a} \} \cup \{ \bold{a} + de_i : 1 \leq i \leq k \}$ for some $\bold{a} \in \{1,2, \dots, N\}^k$ and  $d > 0$,  where $e_1,e_2, \cdots, e_k$ is the standard basis of $\mathbb{R}^k$. We define the maximum size of a $k$-dimensional corner-free subset of  $\{1,2, \dots, N\}^k$ as $c_k(N)$. In this paper, we show that the number of  $k$-dimensional corner-free subsets of the $k$-dimensional grid $\{1,2, \dots, N\}^k$ is at most $2^{O(c_k(N))}$ for infinitely many values of $N$. 
Our main tools for proof are the hypergraph container method and the  supersaturation result for $k$-dimensional corners in sets of size $\Theta(c_k(N))$.
\end{abstract}

\section{Introduction}

In 1975, Szemer\'edi~\cite{Z2} proved that for every real number $\delta > 0 $ and every positive integer $k$, there exists a positive integer $N$ such that every subset $A$ of the set $\{1, 2, \cdots, N \}$ with $|A| \geq \delta N$ contains an arithmetic progression of length $k$. There has been a plethora of research related to Szemer\'edi's theorem mixing methods in many areas of
 mathematics.  Szemer\'edi's original proof is a tour de force of involved  combinatorial arguments.
 There have been now
alternative proofs of Szemer\'edi's theorem by Furstenberg~\cite{FK} using methods from ergodic theory, and by Gowers~\cite{G1} using high order Fourier analysis. The case $k=3$ was proven earlier by Roth~\cite{RO}.\\

A subset $A$ of the set $\{1,2, \dots, N\}$  is said to be $k$-{\it AP-free} if it does not contain an arithmetic progression of length $k$.
We define the maximum size of a $k$-AP-free subset of $\{1,2, \dots, N\}$ as $r_k(N)$. In 1990, Cameron and Erd\H{o}s~\cite{CE} were interested in  counting the number of subsets of the set $\{1,2, \dots, N\}$ which do not contain an arithmetic progression of length $k$ and asked the following question.

\begin{question}[Cameron and Erd\H{o}s~\cite{CE}]
For every positive integer $k$ and $N$, is it true that the number of $k$-AP free subsets of $\{1,2, \dots, N\}$ is 
$2^{(1+o(1)) r_k(N)}?$
\end{question}

Until recently,  research on how to improve the bounds  $r_k(N)$ has been studied by many authors~\cite{BE,BS,EL,OB,G0,G1}. Despite much effort, the difference between the currently known lower and upper bounds of $r_3(N)$ is still quite large. The upper bound has improved gradually over the years, and the current best upper bound is due to Bloom and Sisask~\cite{BS}:

$$ r_3(N) \leq \frac{N}{(\log N)^{1+c}},$$

\noindent where $c>0$ is an absolute constant.\\

\noindent For a lower bound of $r_3(N)$, the configuration of Behrend~\cite{BE}  shows:

$$ r_3(N) = \Omega\Big(\frac{N}{2^{2\sqrt{2}\sqrt{\log_2 N}} \cdot \log^{\frac{1}{4}} N}\Big).$$

\noindent This has been improved by Elkin's modification~\cite{EL} by a factor of $\sqrt{\log n}$.\\

\noindent The currently known lower and upper bounds for $r_k(N)$ are as follows:\\

\noindent Let $m = \lceil \log_2 k\rceil$. For $k \geq 4$, there exist $c_k,c'_k > 0$ such that
$$  c_k \cdot N \cdot (\log N)^{1/2m}  \cdot 2^{-m2^{(m-1)/2(\log n)^{1/m}}}\leq r_k(N) \leq  \frac{N}{(\log \log N)^{c'_k}},$$

\noindent where the lower bound is due to O'Bryant~\cite{OB} and the upper bound is due to Gowers~\cite{G0,G1}.\\

In 2017,  Balogh, Liu, and Sharifzadeh~\cite{BLS} provided a weaker version of  Cameron and Erd\H{o}s's conjecture~\cite{CE} that the number of subsets of the set $\{1,2, \dots, N\}$ without 
an  arithmetic progression of length $k$ is at most $2^{O(r_k(N))}$ for infinitely many values of $N$, which is optimal up to a constant factor in the exponent.\\

A triple of points in the $2$-dimensional grid $\{1,2, \dots, N\}^2$  is called a {\it  corner} if it is of the form  $(a_1,a_2), (a_1+d, a_2), (a_1, a_2+d)$ for some $a_1, a_2 \in \{1,2, \dots, N\}$ and  $d > 0$. In 1974, Ajtai and Szemer\'edi~\cite{AS} discovered that for every number $\delta > 0$,  there exists a positive integer $N$ such that every subset $A$ of the $2$-dimensional grid $\{1,2, \dots, N\}^2$ with $|A| \geq \delta N^2$ contains a corner.  In 1991, F\"{u}rstenberg and Katznelson~\cite{FK1} found that  their more general theorem  implied the result of  Ajtai and Szemer\'edi~\cite{AS}, but did not specify an explicit bound for $N$ as it uses ergodic theory. An easy consequence of their result is the case $k=3$ of Szemer\'edi's theorem, which was first proved by Roth~\cite{RO} using Fourier analysis. 
Afterward, in
2003, Solymosi~\cite{S1} provided a simple proof for Ajtai and Szemer\'edi~\cite{AS} theorem using the Triangle Removal Lemma.  \\

A subset $A$ of the $2$-dimensional grid $\{1,2, \dots, N\}^2$  is called {\it corner-free} 
if it does not contain a corner.  We define the maximum size of corner-free sets in $\{1,2, \dots, N\}^2$ as  $c_2(N)$. 
The problem of  improving the bounds for  $c_2(N)$  has been studied by many authors~\cite{GR1,LS,SH,SH1}. 
The current best lower bound of $c_2(N)$ is due to Green~\cite{GR1}, based on  Linial and Shraibman's construction~\cite{LS}:

$$ \frac{N^2}{2^{(l_1+o(1))\sqrt{\log_2 N}}} \leq c_2(N),$$

\noindent where $l_1 \approx 1.822$. \\

\noindent The current best upper bound of $c_2(N)$ is due to Shkredov~\cite{SH}:
 
 $$ c_2(N) \leq \frac{N^2}{(\log \log N)^{l_2}},$$
 
 \noindent where $l_2 \approx 0.0137$. \\

The higher dimensional analog of a corner in the $2$-dimensional grid $\{1,2, \dots, N\}^2$ is the following.
A subset $A$ of the $k$-dimensional grid $\{1,2, \dots, N\}^k$  is called  $k$-{\it  dimensional corner} if it is
 a set of points of the  
form $\{ \bold{a} \} \cup \{ \bold{a} + de_i : 1 \leq i \leq k \}$  for some $\bold{a} \in \{1,2, \dots, N\}^k$ and  $d > 0$,  where $e_1,e_2, \cdots, e_k$ is the standard basis of $\mathbb{R}^k$.  The following multidimensional version of Ajtai and Szemer\'edi theorem~\cite{AS} was proved by F\"{u}rstenberg, Katznelson~\cite{FK1}, and Gowers~\cite{G2}.\\

\begin{theorem}[~\cite{FK1,G2}]\label{ktheorem}
For every number $\delta > 0$ and  every positive integer $k$,  there exists a positive integer $N$ such that every subset $A$ of the $k$-dimensional grid $\{1,2, \dots, N\}^k$ with $|A| \geq \delta N^k$ contains a $k$-dimensional corner.\\
\end{theorem}

In 1991,   F\"{u}rstenberg and Katznelson~\cite{FK1} showed that their more general theorem  implied Theorem~\ref{ktheorem}, but did not specify an explicit bound  as it uses ergodic theory.
 Later, in 2007, Gowers~\cite{G2} provided the first proof with
explicit bounds and the first proof of Theorem~\ref{ktheorem}  not based on F\"{u}rstenberg's ergodic-theoretic approach. They also proved that Theorem~\ref{ktheorem} implied  the multidimensional Szemer\'edi theorem.\\

Another fundamental result in additive combinatorics is the multidimensional Szemer\'edi theorem, which was demonstrated for the  first time by F\"{u}rstenberg and Katznelson~\cite{FK} using the ergodic method, but provided no explicit bounds.
   In 2007, Gowers~\cite{G2} yielded a combinatorial proof of the multidimensional Szemer\'edi theorem by establishing the Regularity  and Counting Lemmas for the $r$-uniform hypergraph. This is the first proof to provide an explicit bound. Similar results were obtained independently by Nagle, R\" odl, and Schacht~\cite{NRS}.\\

\begin{theorem}[Multidimensional Szemer\'edi theorem~\cite{FK,G2,NRS}]
For every real number $\delta > 0$, every positive integer $k$, and every finite set $X \subset \mathbb{Z}^k$, there exists a positive integer $N$ such that every subset $A$ of the $k$-dimensional grid $\{1,2, \dots, N\}^k$ with $|A| \geq \delta N^k$ contains a subset of the form $\bold{a}+dX$ for some $\bold{a} \in \{1,2, \dots, N\}^k$ and  $d > 0$.\\ \end{theorem}

A subset $A$ of the $k$-dimensional grid $\{1,2, \dots, N\}^k$  is called  $k$-{\it dimensional corner-free} 
if it does not contain a $k$-dimensional corner. We define the maximum size of a $k$-dimensional corner-free subset of  $\{1,2, \dots, N\}^k$ as $c_k(N)$. 
In this paper, we study a natural higher dimensional version of the question of Cameron and Erd\H{o}s, i.e. counting $k$-dimensional corner-free sets in $\{1,2, \dots, N\}^k$ as follows.\\

\begin{question}
For every positive integer $k$ and $N$, is it true that the number of $k$-dimensional corner-free subsets of the $k$-dimensional grid $\{1,2, \dots, N\}^k $is 
$2^{(1+o(1)) c_k(N)}?$\\
\end{question}

In addressing  this question, we show the following theorem. Similar to the results of Balogh, Liu, and Sharifzadeh~\cite{BLS}, despite not knowing the value of the extremal function $c_k(N)$, we can derive a counting result that is optimal up to a constant factor in the exponent.\\
  
  \begin{theorem}\label{main:thm}
The number of $k$-dimensional corner-free subsets of the $k$-dimensional grid $\{1,2, \dots, N\}^k$ is 
$2^{O(c_k(N))}$ for infinitely many values of $N$.\\  \end{theorem}

Our paper is organized as follows.  In Section $2$, we provide the two main tools for  proof: the hypergraph container theorem and supersaturation results for $k$-dimensional corners.
 In Section $3$, we provide  proof of the saturation result 
  for $k$-dimensional corners in sets of size $\Theta(c_k(N))$, which is specified in Section $2$. In Section $4$,  we provide  proof of our main result,
Theorem~\ref{main:thm}.\\

\section{Preliminaries}

\subsection{Hypergraph Container Method}

The hypergraph container method~\cite{BMS,ST} is a very powerful technique for bounding the number of discrete objects  avoiding certain forbidden structures. A graph is $H$-free if it does not have subgraphs that are isomorphic to $H$.
For example, we use the container method when we count the family of 
$H$-free  graphs or the family of sets without $k$ term arithmetic progression. The $r$-uniform hypergraph $\mathcal{H}$ is defined as the pair $(V(\mathcal{H}),E(\mathcal{H}))$ where $V(\mathcal{H})$ is the set of vertices and $E(\mathcal{H})$ is the set of hyperedges that are the $r$-subset of the vertices of $V(\mathcal{H})$. Let $\Gamma(\mathcal{H})$ be a collection of independent sets of   hypergraph $\mathcal{H}$, where the independent set of hypergraph $\mathcal{H}$ is   the set of vertices inducing no hyperedge in $E(\mathcal{H})$.
For a given hypergraph $\mathcal{H}$,  we define the maximum degree of a set of $l$ vertices of $\mathcal{H}$ as
$$\Delta_l(\mathcal{H}) = \max\{ \ d_{\mathcal{H}}(A) : A \subset V(\mathcal{H}),  \ |A| = l \ \}, $$

\noindent where $ d_{\mathcal{H}}(A) $ is the number of hyperedges in $E(\mathcal{H})$ containing the set $A$. \\

 Let $\mathcal{H}$ be an $r$-uniform hypergraph of order $n$ and average degree $d$. For any $ 0 < \tau < 1$, the {\it co-degree}  $ \Delta(\mathcal{H}, \tau)$  is defined as
$$ \Delta(\mathcal{H}, \tau) = 2^{{r \choose 2} -1} \sum_{j=2}^r 2^{-{j-1}\choose 2} \frac{\Delta_j(\mathcal{H})}{{\tau}^{j-1} d}.  $$\\

 In this paper, we use the following  hypergraph container lemma, which contains accurate estimates for the $r$-uniform hypergraph in Corollary 3.6 in~\cite{ST}.\\

 \begin{theorem}[Hypergraph Container Lemma~\cite{ST}]\label{Container1} For every positive integer $r\in \mathbb{N}$, let $\mathcal{H} \subseteq {{V} \choose {r}}$ be an $r$-uniform hypergraph.
 Suppose that there exist $ 0 < \epsilon, \tau < 1/2$  such that
 \begin{itemize}
 \item $\tau < 1/ (200\cdot r\cdot r!^2)$
 \item $ \Delta(\mathcal{H}, \tau) \leq \frac{\epsilon}{12r!}.$\\
 \end{itemize}

 \noindent Then there exist $c =c(r)  \leq 1000 \cdot r \cdot r!^3$ and a collection $\mathcal{C}$ of subsets of $V(\mathcal{H})$  such that the following holds:

 \begin{itemize}
 \item for every independent set $I \in \Gamma(\mathcal{H})$, there exists  $S \in \mathcal{C}$  such that $I \subset S$,
 \item $\log |\mathcal{C}| \leq c \cdot  |V|\cdot  \ \tau \cdot \log(1/\epsilon)\cdot \log (1/\tau)$,
 \item for every $S\in \mathcal{C}$,  $e (\mathcal{H}[S]) \leq \epsilon \cdot e(\mathcal{H}),$
 \end{itemize}
 where $\mathcal{H}[S]$ is a subhypergraph of $\mathcal{H}$ induced by $S$.\\
 \end{theorem}

  Let us consider a $(k+1)$-uniform hypergraph $\mathcal{G}$ encoding the set of all $k$-dimensional corners in  the $k$-dimensional grid $[n]^k$. It means that $V(\mathcal{G}) =[n]^k$ and  the edge set of $\mathcal{G}$ consists of all $(k+1)$-tuples  forming  $k$-dimensional corners. 
Note that  the independent set in   $\mathcal{G}$ is the $k$-dimensional corner-free set in $[n]^k$.
Applying the Hypergraph Container Lemma to the  hypergraph  $\mathcal{G}$ gives the following theorem, which is an important result to prove our  main result, Theorem~\ref{main:thm}. \\

  \begin{theorem}\label{Container1} For every positive integer $k\in \mathbb{N}$, let $\mathcal{G} $ be a $(k+1)$-uniform hypergraph encoding the set of all $k$-dimensional corners in $[n]^k$. 
  Suppose that there exists $ 0 < \epsilon, \tau < 1/2$  satisfying that
 \begin{itemize}
 \item $\tau < 1/ (200\cdot (k+1)\cdot (k+1)!^2)$
 \item $ \Delta(\mathcal{G}, \tau) \leq \frac{\epsilon}{12(k+1)!}.$\\
 \end{itemize}

 \noindent Then there exist $c =c(k+1)  \leq 1000 \cdot (k+1) \cdot (k+1)!^3$ and a collection $\mathcal{C}$ of subsets of $V(\mathcal{G})$  such that the following holds.

    \begin{enumerate}
 \item [(i)]every $k$-dimensional corner-free subset of $[n]^k$  is contained in some  $S \in \mathcal{C}$,  
 \item [(ii)] $\log |\mathcal{C}| \leq c \cdot |V(\mathcal{G})| \cdot \ \tau \cdot \log(1/\epsilon)\cdot \log (1/\tau)$,
 \item [(iii)] for every $S\in \mathcal{C}$,  the number of $k$-dimensional corners in $S$ is at most $\epsilon \cdot e(\mathcal{G})$. \\
 \end{enumerate}
 \end{theorem}

\subsection{Supersaturation Results}

In this section, we present the supersaturation result for $k$-dimensional corners, which is  the second main ingredient for proof of our main result.  A supersaturation result says that sufficiently dense subsets of a given set contain many copies of certain structures.  For the arithmetic progression, the supersaturation result concerned only sets of size linear in $n$ was first demonstrated  by Varnavides~\cite{V}  by showing that any subset of $[n]$ of size $\Omega(n)$ has $\Omega(n^2)$ $k$-APs. In 2008, Green and Tao~\cite{GT} 
obtained the supersaturation result by proving that any subset of $\mathbf{P}_{\leq n}$ of size $\Omega \left(|\mathbf{P}_{\leq n}|\right)$ has $\Theta(n^2 / \log^k n)$ $k$-APs, where $\mathbf{P}_{\leq n}$ is the set of prime numbers up to  $n$.
Later, Croot and Sisask~\cite{CS} provided a quantitative version of Varnavides~\cite{V} by proving that 
for every $1\leq M \leq n$, the number of $3$-AP in $A$ is at least 
$$ \left( \frac{|A|}{n} - \frac{r_3(M)+1}{M}\right)\cdot \frac{n^2}{M^4}.$$\\

	To prove Theorem~\ref{main:thm}, we need the supersaturation result of the minimum value of the number of $k$-dimensional corners for any set $A$ in the $k$-dimensional grid $[n]^k$ of size $\Theta(c_k(N))$. To explain the supersaturation results, we introduce the following definitions. Recall that we define the maximum size of a $k$-dimensional corner-free subset of  the $k$-dimensional grid  $[n]^k$ as $c_k(n)$. 
 Let $\Gamma_k(A)$ denote the number of $k$-dimensional corners in the set $A \subseteq [n]^k$.  
 The following theorem shows that the number of $k$-dimensional corners in any set $A \subseteq [n]^k$ of size constant factor times larger than $c_k(n)$ is superlinear in $n$.  In Section $3$, we provide  proof of  Theorem~\ref{supersaturation_thm}.  \\

 \begin{theorem}\label{supersaturation_thm}
For the given $k\geq 3$, there exist $C' : = C'(k)$ and an infinite sequence $\{ n_i \}_{i=1}^{\infty}$ such that the following holds.
For all $n \in \{ n_i \}_{i=1}^{\infty}$ and any set $A$  in the $k$-dimensional grid $[n]^k$  of size $C'\cdot c_k(n)$, we have
 $$ \Gamma_k(A) \geq \log^{(3k+1)}n \cdot \left( \frac{n^k}{c_k(n)}\right)^{k} \cdot n^{k-1} = \Upsilon(n) \cdot n^k, $$
 where $\Upsilon  (n)  =  \frac{\log^{3k+1} n}{n}  \cdot \left(  \frac {n^k} {c_k(n)} \right)^{k}$.\\
\end{theorem}

\subsubsection{Supersaturation Lemmas}
 In this section, we present more supersaturation results for the minimum value of the number of $k$-dimensional corners  to obtain a superlinear bound in Theorem~\ref{supersaturation_thm}. First, we provide
  the following simple supersaturation result using the greedy algorithm.

 \begin{lemma}\label{lem1}
 For the positive integer $k \geq 2$, let  $A$ be any set  in the $k$-dimensional grid $[n]^k$  of size  $K\cdot c_k(n)$, where $K\geq 2$ is a constant. Then we get
 $$ \Gamma_k(A) \geq (K-1) \cdot c_k(n).$$
 \end{lemma}
 
 \begin{proof}[Proof] We use the greedy algorithm to determine the minimum value of the number of $k$-dimensional corners in a set $A$  of size $K\cdot c_k(n)$, where $K\geq 2$. We consider the following process iteratively. 
 As $|A| > c_k(n)$, there exists 
 a $k$-dimensional corner $C$ in the set $A$. It then updates the set $A$ by removing an arbitrary element from $C$. By repeating  this process $(K-1)\cdot c_k(n)$ times, we have
 $$ \Gamma_k(A) \geq (K-1) \cdot c_k(n).$$
 \end{proof}
 
\noindent Next, we use Lemma~\ref{lem1} to give the following improved supersaturation result.\\

  \begin{lemma}\label{lem2}
 For the positive integer $k \geq 2$, let  $A$ be any set  in the $k$-dimensional grid $[n]^k$  of size at least $K \cdot c_k(n)$, where $K \geq 2$ is a constant. Then we obtain 
 $$ \Gamma_k(A) \geq \left(\frac{K}{2}\right)^{k+1} \cdot c_k(n).$$
 \end{lemma}
 
 \begin{proof} [Proof]Let  $A$ be any set  of  $[n]^k$  and have a size greater than equal to $K\cdot c_k(n)$. 
 We consider the set $S$, which is one of all  subsets of $A$ of size $2 \cdot c_k(n)$. With Lemma~\ref{lem1}, we have $ \Gamma_k(S) \geq c_k(n)$ for every $S$. Therefore we get\\
 
 $$ {{|A|} \choose {2\cdot c_k(n)}}\cdot c_k(n) \  \ \leq \sum_{S\subseteq A, \\ |S|= 2\cdot c_k(n)} \Gamma_k (S) \  \  \leq  \ \Gamma_k(A) \cdot{{|A|-k-1} \choose {2\cdot c_k(n)-k-1}}.  $$ \\

\noindent Then we conclude that
\begin{align*}
 \Gamma_k(A) & \geq \ \  \frac{{{|A|}\choose{2 \cdot c_k(n)}}}{{{|A|-k-1}\choose{2\cdot c_k(n)-k-1}}} \cdot \ c_k(n) \\
  & \geq  \  \ \left( \frac{|A|}{2 \cdot c_k(n)} \right)^{k+1} \cdot c_k(n) \\
  & \geq \ \ \left(\frac{K}{2}\right)^{k+1}  \cdot c_k(n). 
 \end{align*}
 \end{proof}
 
\noindent Note that the bounds of Lemma~\ref{lem1} and Lemma~\ref{lem2} are linear in the set $A$ of $[n]^k$.    In the following lemma, we provide a superlinear bound for the minimum value of the number of $k$-dimensional corners by applying Lemma~\ref{lem2} to the set of carefully chosen $k$-dimensional corners with prime common differences. The following lemma is an important result for proving  the supersaturation result for $k$-dimensional corners in sets of size $\Theta(c_k(N))$ with superlinear bounds, which is specified in Theorem~\ref{supersaturation_thm}.\\

\begin{lemma}\label{lem3}
For the positive integer $k \geq 2$, let  $A$ be any set  in the $k$-dimensional grid $[n]^k$  such that  there exists  a positive constant $M$ satisfying 
$\frac{|A|}{2^{k+1}Mn^{k-1}}$ is sufficiently large and 
$\frac{|A|}{n^k} \geq  \frac{8K \cdot c_k(M)}{M^k}$, where $K\geq 2$ is a constant.  Then we obtain 
 $$ \Gamma_k(A) \geq \frac{|A|^2}{2^{2k+4}} \cdot \frac{(K)^{k+1}\cdot c_k(M)}{M^{k+1}n^{k-1}\log^2 n }.$$
\end{lemma}

\begin{proof}[Proof]
Given the set $A$ of  $[n]^k$,  we let $x= \frac{|A|}{2^{k+1}Mn^{k-1}}$ which is  sufficiently large.
Let $\mathcal{G}_d$ be the set of $M\times \cdots\times M$ grids in $[n]^k$, whose consecutive layers are of distance $d$ apart,  for  a prime  $d \leq x$. Let us consider $\mathcal{G}= \bigcup_{d \leq x} \mathcal{G}_d $.  For any $k$-dimensional corner $C = \{ \bold{a}\} \cup \{ \bold{a} + d'e_i : 1 \leq i \leq k \}$ for some  $\bold{a} \in [n]^k$ and $d'> 0$, where $e_1,e_2, \cdots, e_k$ are the standard bases of $\mathbb{R}^k$, we consider a grid $G \in \mathcal{G}_{d}$  containing $C$. This means that $d$ must be a prime divisor of $d'$. The number of prime divisors of $d'$ is at most $\log d'\leq \log n$, so the number of these choices is at most $\log n$.
Since every corner can occur in at most $(M-1)^k$ grids from each fixed $\mathcal{G}_{d}$ and the length of the corner has at most $\log n$ distinct prime factors, we get

 \begin{align}\label{label1}
\Gamma_k(A) \geq \frac{1}{M^k\cdot \log n} \sum_{G\in \mathcal{G}} \Gamma_k(A \cap G).\\ \nonumber
\end{align}

\noindent Let us  consider $\mathcal{R} \subseteq \mathcal{G}$ consisting of all $G \in \mathcal{G}$    such that $|A \cap G| \geq K\cdot c_k(M)$, where $K\geq 2$ is a constant.  Applying  Lemma~\ref{lem2} to $A \cap G$ gives:

\begin{align}\label{label2}
 \Gamma_k(A \cap G) & \geq \left(\frac{K}{2}\right)^{k+1} \cdot c_k(M).
  \end{align}
 
\noindent for all $G\in \mathcal{R}$. Combining the inequalities (\ref{label1}) and (\ref{label2}), we obtain\\
\begin{align}\label{label3}
\Gamma_k(A) &\geq \frac{1}{M^k \cdot \log n} \sum_{G\in \mathcal{G}} \Gamma_k(A \cap G) \nonumber \\ 
&= \frac{1}{M^k \cdot \log n} \left( \sum_{G\in \mathcal{R}} \Gamma_k(A \cap G) + \sum_{G\in \mathcal{G / R}} \Gamma_k(A \cap G)\right)\nonumber \\
 & \geq |\mathcal{R}|\cdot  \left(\frac{K}{2}\right)^{k+1} \cdot \frac{c_k(M)}{M^k\cdot \log n}. \\ \nonumber
 \end{align}

\noindent Next, let us prove the lower bound for $|\mathcal{R}|$.
 For a prime number $d\leq x = \frac{|A|}{2^{k+1}Mn^{k-1}}$, we define $\zeta_d : = [ (M-1)d+1, n -(M-1)d]^k$.  Then we get the following inequality:
 
\begin{align}
|A \cap \zeta_d| & \geq |A| - 2^kMdn^{k-1} \nonumber \\ & \geq |A|-2^k M n^{k-1} \frac{|A|}{2^{k+1}Mn^{k-1}}  = \frac{|A|}{2}. \nonumber \\  \nonumber
\end{align}

\noindent Note that the number of primes less than or equal to $x$ is at least $\frac{x}{\log x}$ and at most $\frac{2x}{\log x}$ by the Prime Number Theorem.
 Since every $z \in \zeta_d$ appears exactly in the $M^k$ members of $\mathcal{G}_d$, we derive that
 
 \begin{align}\label{label0}
 \sum_{G\in \mathcal{G}} |A \cap G|& =   \sum_{d \leq x} \sum_{G \in \mathcal{G}_d} |A \cap G | \nonumber\\
 & \geq M^k \sum_{d\leq x} |A \cap \zeta_d| \geq M^k \cdot \frac{x}{\log x} \cdot \frac{|A|}{2}.
 \end{align}
 
  \noindent Obviously the inequality  $|\mathcal{G}_d| \leq n^k$  is held for each prime number $d \leq x$. Then we get the following equation:
  
 \begin{align}\label{labelg1}
 |\mathcal{G}|= |\bigcup_{d \leq x} \mathcal{G}_d| \leq \frac{2x}{\log x} \cdot n^k. \\ \nonumber
 \end{align}
 
\noindent  Since   $\mathcal{R} \subseteq \mathcal{G}$ consists of all $G \in \mathcal{G}$    such that $|A \cap G| \geq K\cdot c_k(M)$, using the equation (\ref{labelg1}) we get

\begin{align}\label{label5}
\sum_{G \in \mathcal{G}} |A \cap G| & = \sum_{G \in \mathcal{R}} |A \cap G| + \sum_{G \in \mathcal{G} \backslash \mathcal{R}} |A \cap G| \nonumber\\& \leq M^k |\mathcal{R}| + K\cdot c_k(M)\cdot|\mathcal{G} \backslash \mathcal{R}| \nonumber\\
& \leq M^k|\mathcal{R}| + K \cdot c_k(M) \cdot|\mathcal{G}| \nonumber\\
&  \stackrel{(\ref{labelg1})}\leq M^k|\mathcal{R}| + K \cdot c_k(M) \cdot \frac{2x}{\log x} \cdot n^k. \\ \nonumber
\end{align}

\noindent Using the equations  (\ref{label0}) and (\ref{label5}), we obtain

\begin{align}
|\mathcal{R}| & \stackrel{(\ref{label5})}\geq \frac{1}{M^k} \cdot \left(\sum_{G \in \mathcal{G}} |A \cap G| - K \cdot c_k(M) \cdot \frac{2x}{\log x} \cdot n^k\right) \nonumber\\ & \stackrel{(\ref{label0})}\geq \frac{1}{M^k} \cdot \left( M^k \cdot \frac{x}{\log x} \cdot \frac{|A|}{2}- K \cdot c_k(M) \cdot \frac{2x}{\log x} \cdot n^k\right) \nonumber\\
& = \frac{x}{\log x}\cdot \frac{|A|}{2} - \frac{K\cdot c_k(M)}{M^k}\cdot \frac{2x}{\log x} \cdot n^k \nonumber\\
& = \frac{x}{\log x} \cdot \left( \frac{|A|}{2} - \frac{2K\cdot c_k(M)}{M^k}\cdot n^k\right). \nonumber\\ \nonumber
\end{align}

\noindent From the condition 
$\frac{|A|}{n^k} \geq  \frac{8K \cdot c_k(M)}{M^k}$, we have

\begin{align}\label{label6}
|\mathcal{R}|  
& \geq \frac{x}{\log x} \cdot \left( \frac{|A|}{2} - \frac{2K\cdot c_k(M)}{M^k}\cdot n^k\right) \nonumber\\
&  \geq \frac{x}{\log x} \cdot \left( \frac{|A|}{2} - \frac{1}{4} \frac{|A|}{n^k} \cdot n^k\right) \nonumber\\
& \geq \frac{x}{\log x}\cdot  \frac{|A|}{4} \geq \frac{|A|}{4} \cdot \frac{|A|}{2^{k+1}Mn^{k-1}}\cdot \frac{1}{\log n}. \\ \nonumber
\end{align}

\noindent Using the equations (\ref{label3}) and (\ref{label6}), we conclude that

\begin{align*}
 \Gamma_k(A) &\stackrel{(\ref{label3})}\geq |\mathcal{R}|\cdot (\frac{K}{2})^{k+1}\cdot \frac{c_k(M)}{M^k\log n} \\
 & \stackrel{(\ref{label6})}\geq \frac{|A|^2}{4} \cdot \frac{1}{2^{k+1}Mn^{k-1}}\cdot \frac{1}{\log n}  \cdot (\frac{K}{2})^{k+1}\cdot \frac{c_k(M)}{M^k\log n} \\
 &= \frac{|A|^2}{2^{2k+4}}\cdot \frac{(K)^{k+1}\cdot c_k(M)}{M^{k+1}n^{k-1}\log^2 n}.
\end{align*}
\end{proof}

\section{Proof of Theorem~\ref{supersaturation_thm}}

The supersaturation result of $k$-dimensional corners in sets of size $\Theta(c_k(N))$, which is specified in Theorem~\ref{supersaturation_thm}, is the main tool for  proof of Theorem~\ref{main:thm}. In this section, we prove Theorem~\ref{supersaturation_thm} using  Lemma~\ref{lem3} and the following relationship between $f(n_i)$ and $f(\Lambda(n_i))$ for some infinite sequence $ \{ n_i \}_{i=1}^{\infty}$.\\

\noindent For every $n \in \{ n_i \}_{i=1}^{\infty}$ , we define the following functions: 

$$\Lambda(n) = \frac{n}{\log^{3k+3} n}\cdot \left(  \frac{c_k(n)}{n^k} \right)^{k+3}, \ 
\ \ \ \ \ \ \ \ f(n) = \frac{c_k(n)}{n^k},$$ 

\noindent where $c_k(n)$ is the  maximum size of a $k$-dimensional corner-free subset of $[n]^k$. \\

 \begin{lemma}\label{lem5}  For the given $k\geq 3$, there exist $b : = b(k) > 2^{2k}$ and an infinite sequence $\{ n_i \}_{i=1}^{\infty}$ such that
$$b f(n_i) \geq f(\Lambda(n_i))$$  for all $i \geq 1$. \\
\end{lemma}

First,  we give the following relationship between $f(n)$ and $f(m)$ for any $m < n$, which is  what we need to get Lemma~\ref{lem5}.\\

\begin{lemma}\label{lem4}
For every $m < n$, we obtain $f(n) < 2^k \cdot f(m)$.\\
\end{lemma}

\begin{proof}[Proof]
For every $m <n$, we divide the $k$-dimensional grid $[n]^k$ into consecutive grids of size $m^k$ because the corner-free property is invariant under translation. Since any given $k$-dimensional corner free subset of  $[n]^k$ contains at most $c_k(m)$ elements in each grid of size $m^k$, for any $m <n$ we have\\
$$ c_k(n) \leq \lceil \frac{n}{m} \rceil^k \cdot c_k(m).$$\\

\noindent Since $\frac{1}{n^k} \cdot   \lceil \frac{n}{m} \rceil^k < \frac{2^k}{m^k}$ for every $m <n$, we conclude that
$$ f(n) = \frac{c_k(n)}{n^k} \leq \lceil \frac{n}{m} \rceil^k \cdot \frac{c_k(m)}{n^k} < \frac{2^k}{m^k}\cdot c_k(m) = 2^k \cdot f(m).$$\\
\noindent This completes the proof of Lemma~\ref{lem4}.\\
\end{proof}

To get  Lemma~\ref{lem5}, we also need a lower bound on $c_k(n)$, which follows from Rankin~\cite{RA}'s result that is a generalization of 
Behrend~\cite{BE}'s construction of dense $3$-AP-free subset of integers to the case of arbitrary $k\geq 3$.\\

\begin{lemma}\label{lem0}
For the given $k\geq 2$, there exists $\alpha_k$ such that
$$ \frac{c_k(n)}{n^k}  > 2^{-\alpha_k ({\log n})^{\beta_k}}$$
for all sufficiently large $n$, where $\alpha_k$ is a positive absolute constant that depends only on $k$ and $\beta_k = \dfrac{1}{\lceil \log k \rceil}$.\\
\end{lemma}

\begin{proof}[Proof]
Let us first consider  the case when $k=2$. Let $A$ be the $3$-AP-free subset of $[n]$ with size $n\cdot 2^{-\alpha\sqrt{\log n}}$ from Behrend~\cite{BE}'s construction. We construct a dense $2$-dimensional corner-free subset $B$ of $[n]^2$ of size $\Omega(|A|n)$ as follows: Let $L$ be the collection of all lines of the form $y=x+a$ for every $a\in A$, and $B$ 
be the intersection of $L$ and $[n]^2$. It is easy to see that $|B|=\Omega(|A|n)$. It remains to prove that $B$ is $2$-dimensional corner-free. 
Let us assume otherwise, i.e. there exists
a $2$-dimensional corner in the set $B$, say $(x,y), (x+d,y), (x, y+d)$. Then, depending on the configuration, the three elements $y-x=a_1, y-(x+d)=a_2 $, and $(y+d)-x=a_3$ are all in  the set $A$ forming $3$-AP  with $a_2+a_3=2a_1$. This is a contradiction. Since the case of $k\geq 3$ is similar, the result of Rankin~\cite{RA} is used instead, so details are omitted.\\
\end{proof}

 \noindent Now we use Lemma~\ref{lem4} and Lemma~\ref{lem0} to prove  Lemma~\ref{lem5}.\\

\begin{proof}[Proof of Lemma~\ref{lem5}] Fix $b : = b(k) > 2^{2k}$ a large enough constant. Let us assume otherwise, i.e. there exists $n_0$ for all $n \geq n_0$ satisfying 

\begin{align}\label{assumption1}
f(n) < b^{-1} f(\Lambda(n)).
\end{align} \\
Using Lemma~\ref{lem0}, there exists $\alpha_k$ such that $ f(n) > 2^{-\alpha_k(\log n)^{\beta_k}}$ for every sufficiently large $n$,  where $\beta_k = \dfrac{1}{\lceil \log k \rceil}$ and $\alpha_k$ is a positive absolute constant depending only on $k$.  Using these  $\alpha_k$ and $\beta_k$, for all $x \geq 1$,
we define the decreasing function $g(x)$ as

$$ g(x) = 2^{-(k\alpha_k +3\alpha_k+1)(\log x)^{\beta_k}}.$$ \\
 
 \noindent Then we get the following inequality for every $n \geq n_0$:
 
  \begin{align}\label{new1}
 \Lambda(n) &= \frac{n}{\log^{3k+3} n}\cdot \left(  \frac{c_k(n)}{n^k} \right)^{k+3}\nonumber \\\nonumber
&  = \frac{n}{\log^{3k+3} n}\cdot (f(n))^{k+3}\\ \nonumber
&  \stackrel{Lemma~\ref{lem0}}> \frac{n}{\log^{3k+3} n}\cdot\left( 2^{-\alpha_k(\log n)^{\beta_k}} \right)^{k+3}\\ 
& >  n\cdot 2^{-(k \alpha_k +3\alpha_k+1)(\log n)^{\beta_k}} =  n \cdot g(n).\\\nonumber 
 \end{align}

\noindent From the equation~(\ref{new1}), if we apply Lemma~\ref{lem4} to $\Lambda(n)$ and $n \cdot g(n)$ then we derive 
\begin{align}\label{label7}
f(n) \stackrel{(\ref{assumption1})} < b^{-1}f(\Lambda(n))  \stackrel{Lemma~\ref{lem4}}< b^{-1} 2^k \cdot f(n\cdot g(n)) = \left( \frac{b}{2^k}\right)^{-1} \cdot f(n\cdot g(n)),
\end{align}
\noindent for all $n \geq n_0$.\\

\noindent To prove Lemma~\ref{lem5}, we need the following claim.

\begin{claim}\label{claim1} Let us write $t= \lfloor \frac{1}{2} \frac{(\log n)^{\beta_k}}{k \alpha_k  +3 \alpha_k +1}\rfloor$ with $\alpha_k$ satisfying $ f(n) > 2^{-\alpha_k(\log n)^{\beta_k}}$. Then for all  $n > n_0^{1/(1-\beta_k)}$ we obtain that
$$f(n) < \left(\frac{b}{2^{2k}}\right)^{-j} f\left(n\cdot (g(n))^j\right) $$
\noindent for all $ 1 \leq j \leq t$.
\end{claim}

\begin{proof}[Proof of Claim~\ref{claim1}]
We proceed by induction on $j$.  The base case $j=1$ is done by  the equation (\ref{label7}). Assume that the statement of  Claim~\ref{claim1} holds for every $1\leq j < t$. Now we consider $n'=n\cdot (g(n))^j$ for all $1\leq j < t$. Since $g(n)$ is a decreasing function, for each  $ j < t$ we have

\begin{align}\label{labelnew5}
n' = n\cdot (g(n))^j  > n\cdot (g(n))^t & = n \cdot 2^{-(k \alpha_k  + 3\alpha_k +1)(\log n)^{\beta_k}  \ \cdot  \lfloor \frac{1}{2} \frac{(\log n)^{\beta_k}}{k \alpha_k  +3 \alpha_k +1}\rfloor } \nonumber\\ & = n \cdot \left(\frac{1}{2}\right)^{(k \alpha_k  + 3\alpha_k +1)(\log n)^{\beta_k}  \ \cdot  \lfloor \frac{1}{2} \frac{(\log n)^{\beta_k}}{k \alpha_k  +3 \alpha_k +1}\rfloor } \nonumber \\
& \geq  n \cdot \left(\frac{1}{2}\right)^{(k \alpha_k  + 3\alpha_k +1)(\log n)^{\beta_k}  \ \cdot  \left( \frac{1}{2} \frac{(\log n)^{\beta_k}}{k \alpha_k  +3 \alpha_k +1}\right)} \nonumber \\
& \geq n \cdot \left(\frac{1}{2}\right)^{\frac{1}{2} (\log n)^{2\beta_k}} = n \cdot 2^{-\frac{1}{2} (\log n)^{2\beta_k}} =  n^{1-\beta_k} \  > n_0, 
\end{align}
\noindent for all  $n > n_0^{1/ (1-\beta_k)} \geq n_0$. \\

\noindent Note that $n' > n_0$ in the equation (\ref{labelnew5}).   Then we use the equation (\ref{label7}) to get \\
\begin{align} \label{label8}
f(n') \stackrel{(\ref{label7})}<  \left( \frac{b}{2^k}\right)^{-1} \cdot f(n'\cdot g(n')). \\ \nonumber
\end{align}

\noindent Since $n' <n $ and  $g(n)$ is a decreasing function,  we have  $n' \cdot g(n') > n' \cdot g(n)$.
 Applying Lemma~\ref{lem4} to $n' \cdot g(n')$ and $n' \cdot g(n)$ gives:\\
 
 \begin{align}\label{label9}
 f(n' \cdot g(n')) < 2^k f (n' \cdot g(n)). \\ \nonumber
 \end{align}
 
 \noindent Using the equations (\ref{label8}) and (\ref{label9}), we obtain that \\
 
 \begin{align} \label{label10}
f(n') \stackrel{(\ref{label8})} <  \left( \frac{b}{2^k}\right)^{-1} \cdot f(n'\cdot g(n')) \stackrel{(\ref{label9})}< \left(\frac{b}{2^k}\right)^{-1} \cdot  2^k f (n' \cdot g(n)) = \left(\frac{b}{2^{2k}}\right)^{-1}    f (n' \cdot g(n)), \\ \nonumber
\end{align}

\noindent for all $n > n_0^{1/(1-\beta_k)}$.\\

\noindent  According to the inductive hypothesis, for every $ 1 \leq j <t$, we get 

\begin{align}\label{new2}
f(n) &  \ < \left(\frac{b}{2^{2k}}\right)^{-j} \cdot f(n \cdot (g(n))^j ). \\ \nonumber
\end{align}

\noindent From the equations (\ref{label10}) and (\ref{new2}), for every $ 1 \leq j <t$, we  observe that \\

\begin{align}\label{labelnew7}
f(n) &  \  \stackrel{(\ref{new2})} < \left(\frac{b}{2^{2k}}\right)^{-j} \cdot f(n \cdot (g(n))^j )\nonumber \\ \nonumber
 & = \left(\frac{b}{2^{2k}}\right)^{-j} \cdot f(n') \nonumber \\
 & \stackrel{(\ref{label10})}< \left(\frac{b}{2^{2k}}\right)^{-j} \cdot \left(\frac{b}{2^{2k}}\right)^{-1}    f (n' \cdot g(n))\nonumber \\
 & \ =  \left(\frac{b}{2^{2k}}\right)^{-j-1} \cdot f(n \cdot (g(n))^{j+1}),
\end{align}

\noindent when $n > n_0^{1/(1-\beta_k)}$.\\

\noindent From the equation (\ref{labelnew7}), we see that the  statement of Claim~\ref{claim1} also holds for $ j + 1$. By the Induction axiom, the statement of  Claim~\ref{claim1}  holds  for every $ 1 \leq j \leq t$. This completes the proof of Claim~\ref{claim1}.\\
\end{proof}

Let  $t= \lfloor \frac{1}{2} \frac{(\log n)^{\beta_k}}{k \alpha_k + 3 \alpha_k +1}\rfloor$ be an integer when $\alpha_k$ satisfies the inequality $ f(n) > 2^{-\alpha_k(\log n)^{\beta_k}}$.   Assume that $n > n_0^{1/(1-\beta_k)} \geq n_0$.
 Applying Claim~\ref{claim1}, we get \\
 
\begin{align}\label{label11}
f(n) < \left(\frac{b}{2^{2k}}\right)^{-t} f\left(n\cdot (g(n))^t \right). \\  \nonumber
\end{align}

 Note that 
$n\cdot (g(n))^t 
 \geq  n^{1-\beta_k} 
$ from the equation (\ref{labelnew5}). Applying Lemma~\ref{lem4} to $ n\cdot (g(n))^t $ and $n^{1-\beta_k}$ gives: \\

\begin{align}\label{label12}
f\left( n\cdot (g(n))^t \right) <  2^k \cdot f \left(n^{1-\beta_k} \ \right).\\ \nonumber
\end{align}

\noindent Using the equations ~(\ref{label11}) and ~(\ref{label12}), we draw  the following conclusion.\\

\begin{align} \label{labelnew8}
f(n)  & \stackrel{(\ref{label11})}< \left(\frac{b}{2^{2k}}\right)^{-t} \cdot f\left(n\cdot (g(n))^t \right) \nonumber \\
  & \stackrel{(\ref{label12})} < \left(\frac{b}{2^{2k}}\right)^{-t} \cdot 2^k \cdot f \left(n^{1-\beta_k}\ \right) \nonumber \\
  & \ \leq \left(\frac{b}{2^{2k}}\right)^{-t} \cdot 2^k \nonumber \\ &  \ = 2^k \cdot  \left(\frac{b}{2^{2k}}\right)^{-\lfloor \frac{1}{2} \frac{(\log n)^{\beta_k}}{k\alpha_k  + 3 \alpha_k +1}\rfloor} < 2^{-\alpha_k (\log n)^{\beta_k}},
  \end{align}

  \noindent where $b : = b(k) > 2^{2k}$ is a sufficiently large constant.\\

 \noindent  The equation (\ref{labelnew8}) contradicts the definition of $\alpha_k$. This completes the proof of Lemma~\ref{lem5}.\\
\end{proof}

\noindent Now we use Lemma~\ref{lem3} and Lemma~\ref{lem4} to provide a proof of Theorem~\ref{supersaturation_thm}.\\

\begin{proof}[Proof of Theorem~\ref{supersaturation_thm}]
Let $b(k)$ and an infinite sequence $\{ n_i \}_{i=1}^{\infty}$ obtained from Lemma~\ref{lem5}. For all $n \in \{ n_i \}_{i=1}^{\infty}$, we let
 $A$ be any set  in the $k$-dimensional grid $[n]^k$ of size $8 K \cdot b(k) \cdot c_k(n)$. Using Lemma~\ref{lem5}, we get\\

\begin{align}\label{label13}
\frac{|A|}{n^k} = \frac{8 K \cdot b(k) \cdot c_k(n)}{n^k}  \geq  \frac{8K \cdot c_k(\Lambda(n))}{(\Lambda(n))^k},
\end{align}

\noindent and

\begin{align}\label{label14}
\frac{|A|}{2^{k+1}\cdot \Lambda(n) \cdot n^{k-1}} = \frac{8 K \cdot b(k) \cdot c_k(n)}{2^{k+1}\cdot \Lambda(n) \cdot n^{k-1}}  \geq  8 K \cdot b(k) \cdot \left( \frac{\log^3 n}{2} \right)^{k+1},\\ \nonumber \end{align} 

\noindent where $\Lambda(n) = \frac{n}{\log^{3k+3} n}\cdot \left(  \frac{c_k(n)}{n^k} \right)^{k+3}$. \\ \\

\noindent  Applying Lemma~\ref{lem4} to the the inequality $\Lambda(n) \leq n$ gives:  \\

\begin{align}\label{label15}
\frac{c_k(n)}{n^k} < 2^k \cdot \frac{c_k(\Lambda(n))}{(\Lambda(n))^k}.\\ \nonumber
\end{align}

\noindent From the inequality $ \sqrt{n} \leq \Lambda(n) $, we get

\begin{align}\label{label16}
\frac{n}{\Lambda(n)^2} \leq 1.\\ \nonumber
\end{align}

\noindent From the equations (\ref{label13}) and (\ref{label14}), we can apply Lemma~\ref{lem3}  with $M = \Lambda(n)$ and derive that 

\begin{align} \label{new77}
 \Gamma_k(A) &\geq \frac{|A|^2}{2^{2k+4}} \cdot \frac{(K)^{k+1}\cdot c_k(\Lambda(n))}{(\Lambda(n))^{k+1}\cdot n^{k-1}\cdot \log^2 n} \nonumber \\
  & = \frac{8^2 \cdot K^2 \cdot \left(b(k)\right)^2 \cdot \left(c_k(n)\right)^2}{(\Lambda(n))\cdot \log^2 n} \cdot \frac {c_k(\Lambda(n))}{(\Lambda(n))^k} \cdot \frac{(K)^{k+1}}{ n^{k-1}\cdot  2^{2k+4} }, \\ \nonumber
  \end{align}
  
\noindent where $|A|= 8K\cdot b(k) \cdot c_k(n)$.\\

\noindent The following conclusion is drawn using the equations (\ref{label14}), (\ref{label15}),  (\ref{label16}), and (\ref{new77}):

\begin{align*}
 \Gamma_k(A)   &  \stackrel{(\ref{new77})}\geq \frac{8^2 \cdot K^2 \cdot \left(c(k)\right)^2 \cdot \left(c_k(n)\right)^2}{(\Lambda(n))\cdot \log^2 n} \cdot \frac {c_k(\Lambda(n))}{(\Lambda(n))^k} \cdot \frac{(K)^{k+1}}{ n^{k-1}\cdot  2^{2k+4} } \\
  & \stackrel{(\ref{label15})}\geq \frac{8^2 \cdot K^2 \cdot \left(b(k)\right)^2 \cdot \left(c_k(n)\right)^2}{(\Lambda(n))\cdot \log^2 n} \cdot \frac {c_k(n)}{2^k \cdot n^k} \cdot \frac{(K)^{k+1}}{ n^{k-1}\cdot  2^{2k+4} } \\
   & \stackrel{(\ref{label16})}\geq \frac{ \log^{3k+1} n \cdot (n^k)^{k+3} \cdot n \cdot  8^2 \cdot K^2 \cdot \left(b(k)\right)^2 \cdot (c_k(n))^2}{n^2 \cdot  (c_k(n))^{k+3} \cdot (\Lambda(n))^2 \cdot 2^{2k+2} \cdot n^{2k-2}}  \cdot \frac{(K)^{k+1}\cdot c_k(n)}{ 2^{k+2}} \\
   & \stackrel{(\ref{label14})} \geq \log^{3k+1}n \cdot   \left(\frac{n^k}{c_k(n)}\right)^{k+2} \cdot n^{k-1}\cdot 8^2\cdot  K^2 \cdot (b(k))^2 \cdot  \left( \frac{\log^3 n}{2} \right)^{2k+2}   \cdot \frac{(K)^{k+1}}{ 2^{k+2} } \\
   & \ \geq  \ \log^{3k+1}n \cdot   \left(\frac{n^k}{c_k(n)}\right)^{k} \cdot n^{k-1} = \Upsilon(n) \cdot n^k. \\ \nonumber
  \end{align*}
  
  \noindent This completes the proof of Theorem~\ref{supersaturation_thm}.\\
\end{proof}

\section{Proof of Theorem~\ref{main:thm}}
In this section, we prove the main result Theorem~\ref{main:thm}
using  the hypergraph container method(Theorem~\ref{Container1}) and supersaturation result for $k$-dimensional corners in sets of size $\Theta(c_k(N))$(Theorem~\ref{supersaturation_thm}).

\begin{proof}[Proof of Theorem~\ref{main:thm}]
Let $b(k)$ and the infinite sequence $\{ n_i \}_{i=1}^{\infty}$ obtained from Lemma~\ref{lem5}.
 For every $n \in \{ n_i \}_{i=1}^{\infty}$, we define the following functions:
 
\begin{align*}
\Upsilon  (n) & =  \frac{\log^{3k+1} n}{n}  \cdot \left(  \frac {n^k} {c_k(n)} \right)^{k},  \\
\Psi(n)  &  = \frac{c_k(n)}{n^k} \cdot \frac{1}{\log^3 n},
\end{align*}
\noindent where $c_k(n)$ is the  maximum size of a $k$-dimensional corner-free subset of $[n]^k$. \\

\noindent For sufficiently large $n$, we have
\begin{align}\label{labelnew2}
\Psi(n) < \frac{1}{200 \cdot (k+1)^{2(k+1)}} < \frac{1}{200\cdot \left((k+1)!\right)^2 \cdot(k+1)}, 
\end{align}

\noindent and

\begin{align}\label{labelnew1}
\Upsilon(n) \cdot n \cdot \Psi(n)^{k}  & =  \frac{\log^{3k+1} n}{n}  \cdot \left(  \frac {n^k} {c_k(n)} \right)^{k} \cdot n \cdot \left(\frac{c_k(n)}{n^k} \cdot \frac{1}{\log^3 n} \right)^{k}\nonumber \\
& = \log n \nonumber \\ &  > (k+1)^{3(k+1)}. \\ \nonumber
\end{align}

\noindent Let us consider  $(k+1)$-uniform hypergraph $\mathcal{G}$ encoding the set of all $k$-dimensional corners in $[n]^k$.
For a given hypergraph $\mathcal{G}$, the maximum degree of a set of $j$ vertices of $\mathcal{G}$ is 
$\Delta_j(\mathcal{G}) = \max\{ \ d_{\mathcal{G}}(A) : A \subset V(\mathcal{G}),  \ |A| = j \ \}$, where 
$ d_{\mathcal{G}}(A)$ is the number of hyperedges in $E(\mathcal{G})$ containing the set $A$. 
Then the co-degree  of a $(k+1)$-uniform hypergraph $\mathcal{G}$ of order $n$ and average degree $d$  is written as

\begin{align}\label{new78}
 \Delta(\mathcal{G}, \Psi) &= 2^{{k+1 \choose 2} -1} \sum_{j=2}^{k+1} 2^{-{{j-1}\choose 2}} {\Psi (n)}^{-(j-1)} \cdot \frac{\Delta_j(\mathcal{G})}{ d} \nonumber \\
&  = 2^{{k+1 \choose 2} -1} \sum_{j=2}^{k+1} \beta_j \cdot \frac{\Delta_j(\mathcal{G})}{ d},  
 \end{align}

\noindent where $\beta_j = 2^{-{{j-1}\choose 2}} {\Psi (n)}^{-(j-1)}$ for all $2\leq j \leq k+1$. \\ \\

\noindent Since $\Psi(n) < \frac{1}{200 \cdot (k+1)^{2(k+1)}} < 2^{-3(k+1)} $,  we have
\begin{align}\label{label17}
\frac{\beta_j}{\beta_{j+1}} = \frac{2^{{j}\choose 2} \Psi(n)^j} {2^{{j-1}\choose 2} \Psi(n)^{j-1}} = 2^{j-1} \Psi(n) < 2^{(k+1)} \cdot \Psi(n) < 1,
\end{align}
\noindent for all $ 2 \leq j \leq k-1$.\\

\noindent For the case $j=k$,  we  obtain  the following inequality: \\
\begin{align}\label{label18}
(k-1)(k+1)^2 \cdot \frac{\beta_k}{\beta_{k+1}}  = (k-1)(k+1)^2 \cdot 2^{k-1} \Psi(n)  < 1. \\ \nonumber
\end{align}

\noindent Using the equations (\ref{labelnew1}), (\ref{label17}) and (\ref{label18}), we derive that
 
\begin{align}\label{labelnew3}
\Delta(\mathcal{G}, \Psi)  &  = 2^{{k+1 \choose 2} -1} \sum_{j=2}^{k+1} \beta_j \frac{\Delta_j(\mathcal{G})}{ d} \nonumber\\
& \ \ \leq 2^{{k+1 \choose 2} -1} \left( \sum_{j=2}^{k} \beta_j \frac{(k+1)^2}{ d} + \frac{\beta_{k+1}}{d}\right) \nonumber \\
& \stackrel{(\ref{label17})}\leq 2^{{k+1 \choose 2} -1} \left( (k-1) \cdot \beta_k \cdot \frac{(k+1)^2}{ d} + \frac{\beta_{k+1}}{d}\right) \nonumber \\
& \stackrel{(\ref{label18})}\leq 2^{{k+1 \choose 2} -1} \left(  \frac{2\beta_{k+1}}{d}\right) = \frac{2^k}{d\cdot (\Psi(n))^k} \nonumber \\
 &\ \ \leq \frac{(k+1)^{k+1}}{n\cdot (\Psi(n))^k} \stackrel{(\ref{labelnew1})}< \frac{\Upsilon(n)}{12\cdot (k+1)!}. \\ \nonumber \end{align}

 From the equations (\ref{labelnew2}) and (\ref{labelnew3}), we can apply the Hypergraph Container Lemma (Theorem~\ref{Container1}) on  the hypergraph $\mathcal{G}$ with $ \epsilon = \Upsilon(n),  \tau = \Psi(n)$ as a function of $n$
 to get the collection $\mathcal{C}$ of containers such that all $k$-dimensional corner-free subsets of the $k$-dimensional grid $[n]^k$ are contained in some container in $\mathcal{C}$.\\

\noindent Using Theorem~\ref{Container1}, there exist $c = c(k+1) \leq 1000 \cdot (k+1) \cdot ((k+1)!)^3 $ and a collection 
 $\mathcal{C}$
  of containers such that the followings hold:\\

 \begin{itemize}
  \item for every  $k$-dimensional corner free subset of the $k$-dimensional grid $[n]^k$ is contained in some container in $\mathcal{C}$,
  \item $\log |\mathcal{C}| \leq c \cdot  n  \cdot \Psi(n) \cdot \log \frac{1}{\Upsilon(n)}\cdot  \log \frac{1}{\Psi(n)}$,
 \item for every container $A \in \mathcal{C}$ the number of $k$-dimensional corners in $A$ is at most $\Upsilon(n) \cdot n^k$.\\
 \end{itemize}

 \noindent The definitions of $\Upsilon(n)$ and $\Psi(n)$ give the following inequality:
 
 \begin{align}\label{label22}
 \log \frac{1}{\Upsilon(n)}\cdot  \log \frac{1}{\Psi(n)} & = \log \left(   \frac{n}{\log^{3k+1} n}  \cdot \left( \frac{c_k(n)}{n^k} \right)^{k}   \right) \cdot \log \left( \frac{n^k}{c_k(n)} \cdot {\log^3 n} \right) \nonumber\\
 &  \leq \log n \cdot \left((k+3)\log n \right)  = (k+3) \left( \log n \right)^2.\\ \nonumber
 \end{align}

 \noindent Using the equation (\ref{label22}) for the collection $\mathcal{C}$ of containers gives:  \\
 
 \begin{align}\label{label19}
 \log |\mathcal{C}| & \leq c \cdot  n  \cdot \Psi(n) \cdot \log \frac{1}{\Upsilon(n)}\cdot  \log \frac{1}{\Psi(n)}  \nonumber \\
 &\leq 1000 \cdot (k+1) \cdot ((k+1)!)^3 \cdot n  \cdot \Psi(n) \cdot \log \frac{1}{\Upsilon(n)}\cdot  \log \frac{1}{\Psi(n)}\nonumber \\ 
  & \stackrel{(\ref{label22})} \leq 1000 \cdot (k+1) \cdot ((k+1)!)^3 \cdot n  \cdot \frac{c_k(n)}{n^k} \cdot \frac{1}{\log^3 n} \cdot (k+3) \left( \log n \right)^2 = o(c_k(n)).\\ \nonumber
 \end{align}

 \noindent Note that 
 for every container $A \in \mathcal{C}$, the number of $k$-dimensional corners in $A$ is at most $\Upsilon(n) \cdot n^k$.
  Now applying Theorem~\ref{supersaturation_thm} gives:\\
  
 \begin{align}\label{label20}
 |A| < C' \cdot c_k(n), \\ \nonumber
 \end{align}
 \noindent  for every container $A \in \mathcal{C}$.\\
 
 \noindent Since every $k$-dimensional corner free subset of the $k$-dimensional grid $[n]^k$ is contained in some container in $\mathcal{C}$, we conclude that the number of $k$-dimensional corner free subsets of $[n]^k$ is at most \\
 
 \begin{align*}
 \sum_{A\in \mathcal{C}} 2^{|A|} & \leq |\mathcal{C}| \cdot \max_{A \in \mathcal{C}} 2^{|A|}  \\
 & \stackrel{(\ref{label19})\& (\ref{label20})} < 2^{o(c_k(n))} \cdot 2^{C' \cdot c_k(n)}  = 2^{O(c_k(n))}, \\ 
 \end{align*}
 \noindent using the equations (\ref{label19}) and (\ref{label20}). This completes the proof of Theorem~\ref{main:thm}. \\
 \end{proof}

{\bf Acknowledgment.}
I would like to thank  Dong Yeap Kang and Hong Liu for their helpful discussions. I would particularly like to thank Hong Liu for providing many helpful comments.


\begin{thebibliography}{99}


\bibitem{AS} M. Ajtai and E. Szemer\'{e}di, \textit{Sets of lattice points that form no squares}, Studia Scientiarum Mathematicarum Hungarica. \textbf{9} (1974) 9-11.
 
\bibitem{BLS} J. Balogh, H. Liu, and  M. Sharifzadeh,  \textit{The number of subsets of integers with no $k$-term arithmetic progression}, International Mathematics Research Notices \textbf{20} (2017) 6168-6186.  

\bibitem{BMS} J. Balogh, R. Morris and W. Samotij, \textit{Independent sets in hypergraphs}, J. American Math. Soc. \textbf{28} (2015), 669-709. 



\bibitem{BE} F.A. Behrend, \textit{On sets of integers which contain no three terms in arithmetical progression}, Proc. Nat. Acad. Sci. U.S.A. \textbf{2} (1946) 331-332.

\bibitem{BS} T.F. Bloom and O. Sisask, \textit{Breaking the logarithmic barrier in Roth's theorem on arithmetic progressions}, preprint (2021), arXiv:2007.03528.

\bibitem{CE} P. Cameron and P. Erd\H{o}s, \textit{On the number of sets of integers with various properties}, in Number Theory (R.A. Mollin, ed.), Walter de Grnyter, Berlin, (1990) 61-79.


\bibitem{CS} E. Croot and O. Sisask, \textit{A new proof of Roth's theorem on arithmetic progressions}, Proceedings of the American Mathematical Society \textbf{137} (2009) 805-809.




\bibitem{EL} M. Elkin, \textit{An improved construction of progression-free sets}, Israel Journal of Mathematics \textbf{184} (2011) 93-128.



\bibitem{FK} H. F\"{u}rstenberg and Y. Katznelson, \textit{An ergodic Szemer\' edi theorem for commuting transformations}, J. Analyse Math. \textbf{34} (1978) 275-291.

\bibitem{FK1} H. F\"{u}rstenberg and Y. Katznelson, \textit{A density version of the Hales-Jewett theorem}, J. Analyse Math. \textbf{57} (1991) 64-119.


\bibitem{G0} W. T. Gowers,  \textit{A new proof of Szemer\' edi's theorem for progressions of length four}, Geom. Func. Anal. \textbf{8} (1998) 529-551.

\bibitem{G1} W. T. Gowers,  \textit{A new proof of Szemer\' edi's theorem}, Geom. Func. Anal. \textbf{11} (2001) 465-588.

\bibitem{G2} W. T. Gowers,  \textit{Hypergraph regularity and the multidimensional Szemer\' edi theorem}, Annals of Mathematics \textbf{166} (2007) 897-946.

\bibitem{GR1} B. Green, \textit{Lower bounds for corner-free sets}, preprint (2021), arXiv:0710.3032. 

\bibitem{GT} B. Green and T. Tao, \textit{The primes contain arbitrarily long arithmetic progressions}, Annals of Mathematics \textbf{167} (2008) 481-547.


\bibitem{LS} N. Linial and A. Shraibman, \textit{Larger corner-free sets from better NOF exactly-N protocols}, Discrete Analysis \textbf{19} (2021) 9 pp, arXiv:2102.00421.



\bibitem{NRS} B. Nagle, V. R\"odl, and M. Schacht, \textit{The counting lemma for regular $k$-uniform hypergraphs}, Random Structures and Algorithms \textbf{28} (2006) 113-179.


\bibitem{OB} K. O'Bryant, \textit{Sets of integers that do not contain long arithmetic progressions}, Electronic Journal of Combinatorics \textbf{18(1)} (2011) P59.

\bibitem{RA} R.A. Rankin, \textit{Sets of integers containing not more than a given number of terms in arithmetical progression}, Proc. Roy. Soc. Edinburgh Sect. A \textbf{65} (1960)/(1961) 332-344.

\bibitem{RO} K. Roth, \textit{On certain sets of integers}, J. London Math. Soc. \textbf{28} (1953) 245-252.

\bibitem{ST} D. Saxton and A. Thomason, \textit{Hypergraph containers}, Invent. Math. \textbf{201} (2015) 925-992.


\bibitem{SH} I.D. Shkredov, \textit{On a generalization of Szemer\'{e}di's theorem} Proceedings of the London Mathematical Society \textbf{93(3)} (2006) 723-760.

\bibitem{SH1} I.D. Shkredov, \textit{On a problem of Gowers}, Izv. Ross. Akad. Nauk Ser. Mat. \textbf{70(2)} (2006) 179-221,
arXiv:math/0405406v1.


\bibitem{S1} J. Solymosi,  \textit{Note on a generalization of Roth's theorem}, Discrete and Computational Geom. Algorithms Combin. \textbf{25} (2003) 825-827.



\bibitem{Z1} E. Szemer\' edi, \textit{Integer sets containing no arithmetic progressions}, Acta Math. Hungar. \textbf{56} (1990) 155-158.

\bibitem{Z2} E. Szemer\' edi, \textit{On the sets of integers containing no $k$ elements in arithmetic progression}, Acta Arith. \textbf{27} (1975) 199-245.

\bibitem{V} P.Varnavides, \textit{On certain sets of positive density}, J.London Math.Soc. \textbf{34} (1959) 358-360.



\end{thebibliography}
\end{document}